\documentclass[psamsfonts,reqno,12pt]{amsart}
\usepackage{mathrsfs}
\usepackage{amsfonts}
\usepackage{stmaryrd}
\usepackage{amscd,amssymb,verbatim}
\usepackage{amssymb,amsmath,amsthm,epsfig}
\newtheorem{theorem}{Theorem}[section]
\newtheorem{lemma}[theorem]{Lemma}
\newtheorem{prop}[theorem]{Proposition}

\theoremstyle{definition}
\newtheorem{definition}[theorem]{Definition}
\newtheorem{example}[theorem]{Example}

\theoremstyle{remark}

\numberwithin{equation}{section}

\newcommand\N{\mathbb{N}}

\newcommand\J{\mathbb{J}}

\newcommand\cH{\mathcal{H}}

\def\sideremark#1{\ifvmode\leavevmode\fi\vadjust{\vbox
to0pt{\vss \hbox to 0pt{\hskip\hsize\hskip1em
\vbox{\hsize2cm\tiny\raggedright\pretolerance10000
\noindent#1\hfill}\hss}\vbox to8pt{\vfil}\vss}}}

\begin{document}

\title[Hilbert-Schauder Frame Operators]{Hilbert-Schauder Frame Operators}

\author{Rui Liu}
\address{Department of Mathematics and LPMC, Nankai University, Tianjin 300071, P.R. China}
\email{ruiliu@nankai.edu.cn}
\begin{abstract}
We introduce a new concept of frame operators for Banach spaces we call a Hilbert-Schauder frame operator. This is a hybird between standard frame theory for Hilbert spaces and Schauder frame theory for Banach spaces. Most of our results involve basic structure properties of the Hilbert-Schauder frame operator. Examples of Hilbert-Schauder frames include standard Hilbert frames and classical bases of $\ell_p$ and $L^p$-spaces with $1< p \le 2$. Finally, we give a new isomorphic characterization of Hilbert spaces.
\end{abstract}

\thanks{This work was done while the author were visiting the University of Texas at Austin and Texas A\&M University, and in the Linear Analysis Workshop at Texas A\&M University which was funded by NSF Workshop Grant for which Dr. David Larson was a co-PI of that grant. This work is also supported by the NSFC grant
11001134 and 11126250, the Fundamental Research Funds for the
Central Universities, and the Tianjin
Science \& Technology Fund 20100820.}

\date{}



\maketitle

\section*{Introduction}
In 1946, Gabor \cite{Ga} introduced a fundamental approach to signal decomposition
in terms of elementary signals. In 1952, while addressing some difficult problems
from the theory of nonharmonic Fourier series, Duffin and Schaeffer \cite{DS} abstracted
Gabor¡¯s method to define frames for a Hilbert space. For some reason the work
of Duffin and Schaeffer was not continued until 1986 when the fundamental work
of Daubechies, Grossman and Meyer \cite{DGM} brought this all back to life, right at
the dawn of the ¡°wavelet era¡±. Today, the theory of frames in Hilbert spaces presents a central tool in
mathematics and engineering, and has developed rather rapidly in the
past decade. The motivation has come from applications to signal
analysis, as well as from applications to a wide variety of areas of
mathematics, such as operator theory
\cite{HL} and Banach space theory \cite{CHL}.

In 1991, Gr\"{o}chenig \cite{Gr} generalized Hilbert frames to Banach spaces and introduced atomic decompositions and Banach frames. Han and Larson \cite{HL} defined a Schauder frame for a Banach space to be a compression of a Schauder basis for a Banach space. In \cite{CHL}, Casazza,
Han and Larson gave and studied various definitions of frames for Banach spaces including the Schauder frame. In 2009, Casazza, Dilworth, Odell, Schlumprecht and Zs$\acute{\textrm{a}}$k \cite{CDOSZ} studied the coefficient quantization for Schauder frames in Banach spaces. In \cite{CL}, Carando and Lassalle considered the duality theory
for atomic decompositions. Concentrating
on Schauder frames independent of the associated bases,
the author \cite{Li} gave out the concepts of minimal and maximal associated bases with
respect to Schauder frames, closely connected them with the
duality theory, and extended known James' results \cite{Ja} on
unconditional bases to unconditional frames. In \cite{LZ}, the author and Zheng gave an characterization of Schauder frames which are near-Schauder bases, which generalized Holub's results \cite{Ho} from Hilbert frames to Schauder frames. In \cite{CLS}, Carando, Lassalle and Schmidberg considered the reconstruction formula for Banach frames, extended and improved some James' type duality results in \cite{CL,Li}. Recently, Larson, Han, Liu and the author \cite{HLLL} developed elements of a general dilation theory for operator-valued measures and bounded linear maps between operator algebras that are not necessarily completely-bounded, and proved main results by extending and generalizing some known results from the theory of frames and framings. In \cite{BDL}, Beanland, Freeman and the author proved that the upper and lower estimates theorems for finite dimensional decompositions of Banach spaces can be extended and modified to Schauder frames, and gave a complete characterization on duality for Schauder frames.
Recently as well, a continuous version of classical dilation theorem in \cite{HL} was given for vector bundles and Riemannian
manifolds \cite{FPWW}.

From \cite{HL,CHL}, one thing worthwhile to note is that the notion of Hilbert frame transform makes perfect
sense in more general Banach spaces, although the standard frame operator is necessarily a Hilbert space concept. Proposition 1.10 in \cite{HL} forces a similarity between a frame and its canonical dual, so there is an isomorphism between the underlying space and its dual space, and most Banach
spaces are not isomorphic to their dual spaces. Actually, the essential reason here is that the standard frame operator is invertible, which is not necessary for general Banach spaces.

In this paper, we introduce a new concept in frame theory for Banach spaces we call a Hilbert-Schauder frame. This is a hybird between standard Hilbert space frame theory and Schauder frames for Banach spaces. Most of our results involve properties of the associated Hilbert-Schauder frame operator. Examples of unconditional Hilbert-Schauder bases include classical bases of $\ell_p$ and $L^p[0,1]$ with $1< p \le 2$, while $\ell_q$ with $2<q<\infty$ has no Hilbert-Schauder frame. Finally, following the idea in \cite{HL}, we give a new isomorphic characterization of Hilbert spaces.

Throughout this paper we only consider the real case for convenience. The complex
case is more complicated, because we will need the concept of antilinear dual space, denoted by $\overline{X^*}$, to extend the notions of self-adjointness and positivity into $B(X,\overline{X^*})$.
For more information, please see \cite{Pi}.

%

\section{Preliminaries}

\begin{definition} Let $\cH$ be a Hilbert space. A sequence
$\{f_j\}_{j\in\J}$ in $\cH$ is called a \emph{(standard) Hilbert frame} of $\cH$ if
there are $0<a\le b<\infty$ such that
\begin{equation*}
a\|x\|^2\le \sum_{j\in\J} |\langle x, f_n\rangle|^2\le b\|x\|^2 \quad
\mbox{ for all } x\in\cH.
\end{equation*}
\end{definition}

For a Hilbert frame $\{f_j\}_{j\in\J}$ of $\cH$, we consider the operator
$A:\cH\rightarrow\ell_2$ with $x\mapsto \{\langle x, f_j \rangle\}_{j\in\J},$
Its joint
$A^*:\ell_2(\J)\rightarrow\cH$ with $\{a_j\}_{j\in\J}\mapsto\sum_{j\in\J} a_j f_j$
and their product
$$S=A^*A:\cH\rightarrow\cH, \quad x\mapsto\sum_{j\in\J}\langle x, f_j \rangle f_j.$$
Since
$$a\|x\|^2\le\sum_{j\in\J}|\langle x, f_j\rangle|^2
=\big\langle\sum_{j\in\J}\langle x, f_j \rangle f_j,x\big\rangle =\langle S
x,x\rangle\le b\|x\|^2.$$ $S$ is a positive and invertible operator
with $a\, \mathrm{Id}_\cH\le S\le b\, \mathrm{Id}_\cH$ and thus,
\begin{equation}\label{eq:2}x=S^{-1}Sx=\sum_{j\in\J}\langle x,f_j\rangle S^{-1}f_j.
\end{equation}
The operator $S$ is the standard Hilbert frame operator.

For the introduction to the theory of Hilbert frames we refer the
reader to \cite{Ca1} and \cite{Ch1}. We follow
\cite{CHL,HL,CDOSZ,Li,LZ} for the theory of Schauder frames
in Banach spaces.

\begin{definition}
Let $X$ be a Banach space. A sequence $\{x_j,f_j\}_{j\in\J}$ in $X\times X^*$
is called a \emph{Schauder frame} of $X$ if
\begin{equation}\label{eq:5}x=\sum_{j\in\J} \langle x, f_j\rangle
x_j \quad \mbox{ for all } x\in X.\end{equation}
\end{definition}

\section{Hilbert-Schauder Frame Operators}

\begin{definition}Let $X$ be a
separable Banach space. A bounded linear operator
$S:X\rightarrow X^*$ is called a \emph{Hilbert-Schauder frame operator}, or \emph{HSf-operator} for brevity, if there is a Schauder frame $\{x_j,f_j\}_{j\in\J}$ of $X$ such that $S(x_j)=f_j$ for all $j\in\J.$

A Schauder frame $\{x_j,f_j\}_{j\in\J}$ of $X$ is called a \emph{Hilbert-Schauder
frame}, or \emph{HS-frame} for brevity, if there is a bounded linear operator
$S:X\rightarrow X^*$ such that $S(x_j)=f_j$ for all $j\in\J$.
%
\end{definition}

\begin{definition}
Let $X$ be a Banach space and $T\in B(X,X^*)$.
We say that
\begin{enumerate}
\item[(i)] $T$ is \emph{self-adjoint} if $T^*|_X=T$;
\item[(ii)] $T$ is \emph{positive}
if $(Tx)(x)\ge 0$ for all $x\in X$.
\end{enumerate}
\end{definition}
\begin{prop}\label{pp:1}
Every HSf-operator is self-adjoint, positive,
and injective.
\end{prop}
\begin{proof} Let $S$ be a HSf-operator of a Banach space $X$ with the Schauder frame $\{x_j,f_j\}_{j\in\J}$.
To get that $S$ is self-adjoint, it is sufficient to
prove that $(S x)(y)=(S y)(x)$ for all $x,y\in X$ as follows
\begin{eqnarray*}
(S x)(y)&=&\Big(S \sum_j f_j(x)x_j\Big)(y)
=\Big(\sum_j f_j(x)f_j\Big)(y)
=\sum_j f_j(x) f_j(y)\\
&=&\Big(\sum f_j(y)f_j\Big)(x)
=(S y)(x)=(S^*|_X x)(y).
\end{eqnarray*}
$S$ is positive, because for all $x\in X$ we have
\begin{eqnarray}\label{eq:1}
(S x)(x)&=&\Big(S \sum_j f_j(x) x_j\Big)(x)
=\Big(\sum_j f_j(x) f_j\Big)(x)=\sum_j |f_j(x)|^2\ge0.
\end{eqnarray}
If $S(x)=0$, then $(Sx)(x)=0$. By (\ref{eq:1}), we have $0=(Sx)(x)=\sum_j |f_j(x)|^2$, that is, $f_j(x)=0$ for all $j\in\J$. Thus, $x=\sum_j f_j(x) x_j=0$. It
follows that $S$ is injective.
\end{proof}

\begin{lemma}
Let $X$ be a separable Banach space and $\{x_j,f_j\}_{j\in\J}$ be a HS-frame of $X$ with
the HSf-operator $S$. Then 
$\sum_{j\in\J}|f_j(x)|^2\le\|S\| \|x\|^2$ for all $x\in X.$
\end{lemma}
\begin{proof}
By (\ref{eq:1}), we have
$
(S x)(x)=\sum_j|f_j(x)|^2.
$
Thus, $\sum_{j}|f_j(x)|^2=(S x)(x)\le\|S\| \|x\|^2$ for all $x\in X$.
\end{proof}

Thus, the linear operator
$$A:X\rightarrow\ell_2(\J), \ \ x\mapsto\sum_j f_j(x) e_j$$ is
well-defined and
bounded with $\|A\|\le\sqrt{\|S\|}$. $A$ is called the
\emph{Hilbert-Schauder analysis operator}. Its
adjoint operator $A^*$ is given by
$$A^*:\ell_2(\J)\rightarrow X^*, \ \ \sum_j a_j e_j\mapsto\sum_j a_j f_j.$$
$A^*$ is called the \emph{Hilbert-Schauder pre-frame operator}.

\begin{prop}\label{pp:2}
By composing $A$ and $A^*$, we obtain that the HSf-operator $$S=A^*A.$$
Thus, the HSf-operator $S$ factors through $\ell_2(\J)$.
\end{prop}
\begin{proof}
For all $x\in X$, we have
\begin{eqnarray*}A^*A(x)=A^*\big(\sum_j f_j(x)
e_j\big)=\sum_j f_j(x)f_j=S(x).\end{eqnarray*}
\end{proof}

Recall that for Schauder frames, one sequence does not uniquely determine the other, because
it is a redundant system not a stable basis \cite{Li,LZ}. The following
propositions will show that for HS-frames things are different: one sequence uniquely determines the other with
respect to the HSf-operator.

\begin{prop}Let $X$ be a Banach space.
Suppose $\{x_j,f_{1,j}\}_{j\in\J}$ and $\{x_j,f_{2,j}\}_{j\in\J}$ are both HS-frames of $X$ with the HSf-operators $S_1$ and $S_2$, respectively. Then
$S_1=S_2$ and $f_{1,j}=f_{2,j}$ for all $j\in\J$.
\end{prop}
\begin{proof} It is sufficient to prove that $S_1=S_2$.
For all $x,y\in X$ we have
\begin{eqnarray*}
(S_1x)(y)&=&\Big(S_1 \sum_j f_{2,j} (x) x_j\Big)(y)
=\Big(\sum_j f_{2,j}(x) f_{1,j}\Big)(y)\\
&=&\sum_j f_{2,j}(x) f_{1,j}(y)=\Big(\sum_j f_{1,j}(y) f_{2,j}\Big)(x)=\Big(S_2
\sum_j f_{1,j}(y) x_j\Big)(x)\\&=&(S_2y)(x)=(S_2^*|_X x)(y).
\end{eqnarray*}
Then, by Proposition \ref{pp:1}, we obtain that $S_1=S_2^*|_X=S_2.$
\end{proof}


\begin{prop} Let $X$ be a Banach space.
Suppose $\{x_{1,j},f_j\}_{j\in\J}$ and $\{x_{2,j},f_j\}_{j\in\J}$ are both HS-frames of $X$ with the HSf-operators $S_1$ and $S_2$, respectively. Then
$S_1=S_2$ and $x_{1,j}=x_{2,j}$ for all $j\in\J$.
%
\end{prop}
\begin{proof} For all $x\in X$ we have
\begin{eqnarray*}
S_1(x)=S_1\Big(\sum_j f_j(x)x_{1,j}\Big)=\sum f_j(x)f_j=S_2\Big(\sum f_j(x)x_{2,j}\Big)=S_2(x).
\end{eqnarray*}
Then, by Proposition \ref{pp:1}, $S_1=S_2$ and $x_{1,j}=x_{2,j}$ for all $j\in\J$.
\end{proof}

Actually, HS-frames have a better locally duality property to essentially establish its
advantage over Schauder frames.
\begin{prop}\label{pp:7}
Let $\{x_n,f_n\}$ be a HS-frame of $X$ with a HSf-operator $S$. Then $\{f_n,x_n\}$
is a Schauder frame for the closure of \,$\mathrm{span}\{f_n\}$.
\end{prop}
\begin{proof} For each $f_n$, by Proposition \ref{pp:1}, we have
\begin{eqnarray}\label{eq:6} f_n&=&S(x_n)=S(\sum_j f_j(x_n)x_j)=\sum_j (S x_j)(x_n)\cdot S(x_j)\nonumber\\
&=&\sum_j (S^* x_n)(x_j) f_j=\sum_j (S x_n)(x_j) f_j=\sum_j f_n(x_j)f_j.
\end{eqnarray}
By proposition 2.8 in \cite{Li}, the space
$$Y=\Big\{f\in X^*: f=\|\cdot\|-\lim_{n\to\infty}\sum_{j=1}^n f(x_j) f_j\Big\},$$
is a norm closed subspace of $X^*$.
Then, by (\ref{eq:6}), we get
$\overline{\text{\rm span}}(f_n:n\in\N)\subset Y$. On the other hand, it
is clear from the definition of $Y$ that
$Y\subset\overline{\text{\rm span}}(f_n:n\in\N)$. Therefore,
$Y=\overline{\textrm{span}}(f_n:n\in\N).$ Thus, $\{f_n,x_n\}$
is a Schauder frame of $\overline{\textrm{span}}(f_n:n\in\N).$
\end{proof}

However, it is
false for Schauder frames. The following example is an unconditional and
semi-normalized Schauder frame $\{x_n,f_n\}$ of $\ell_1$ for which
$\{f_n,x_n\}$
is not a Schauder frame of $\overline{\textrm{span}}(f_n:n\in\N)$.
\begin{example}\cite{Li}
Let $(e_n)$ denote the usual unit vector basis of $\ell_1$ and let
$(e_n^*)$ be the corresponding coordinate functionals, and set
${\bf 1}=(1,1,1, .\, .\, .\, )\in \ell_\infty$. Then define a
sequence $(x_n,f_n)\subset \ell_1 \times \ell_\infty$ by putting
$x_{2n-1}=x_{2n}=e_n$ for all $n\in\N$ and
$$
f_n=\left\{%
\begin{array}{ll}
    {\bf 1}, & \hbox{if $\!n=\!1$;} \\
    e_1^*-{\bf 1}, & \hbox{if $n\!=\!2$;} \\
    e_k^*-e_1^*/2^k, & \hbox{if $n\!=\!2k-1$ for $k\in\N\setminus\{1\}$;} \\
    e_1^*/2^k, & \hbox{if $n\!=\!2k$ for $k\in\N\setminus\{1\}$.} \\
\end{array}%
\right.$$
\end{example}
\noindent Actually, we have ${\bf 1}\neq \|\cdot\|-\lim_{n\to\infty}\sum_{j=1}^n {\bf 1}(x_j) f_j$.
We leave the detail to the reader.

Now we give some important examples of HS-frames.
\begin{example}
Every standard Hilbert frame operator is an example of a HSf-operator by formula (\ref{eq:2})
$$x=\sum_{j\in\J}\langle x,f_j\rangle S^{-1}f_j=\sum_{j\in\J}\langle x,S(S^{-1}f_j)\rangle S^{-1}f_j.$$
\end{example}

\begin{definition}Let $X$ be a
separable Banach space. A bounded linear operator
$S:X\rightarrow X^*$ is called a \emph{Hilbert-Schauder basis operator}, or \emph{HSb-operator} for brevity, if there is a Schauder basis $\{z_j,z_j^*\}_{j\in\J}$ of $X$ such that $S(z_j)=z_j^*$ for all $j\in\J.$

A Schauder basis $\{z_j,z_j^*\}_{j\in\J}$ of $X$ is called a \emph{Hilbert-Schauder
basis}, or \emph{HS-basis} for brevity, if there is a bounded linear operator
$S:X\rightarrow X^*$ such that $S(z_j)=z_j^*$ for all $j\in\J$.
\end{definition}

\begin{prop}\label{pp:3}
The unit vector basis of $\ell_p$ with $1\le p\le 2$ is an unconditional HS-basis.
\end{prop}
\begin{proof}
Let $\{e_n\}$ be the unit vector basis of $\ell_p$ and $\{e_n^*\}$
be the biorthogonal functionals in the dual space, that is, the unit
vector basis of $\ell_q$ with $1/p+1/q=1.$ Then $\{e_n,e_n^*\}$ is
an unconditional Schauder basis of $\ell_p$. Moreover, since $1<p\le
2\le q<\infty$, we have
$(\sum_n |a_n|^q)^{1/q}\le(\sum_n |a_n|^p)^{1/p}$ for all scalars $\{a_n\}.$
Thus, the operator $$S:\ell_p\rightarrow\ell_q, \quad
\sum_n a_n e_n \mapsto \sum_n a_n e_n^*\quad \mbox{for all } \{a_n\}\in\ell_p$$ is well-defined and
bounded with norm $\|S\|=1.$ Clearly, $S(e_n)=e_n^*$ for all
$n\in\N.$ Thus, $\{e_n,e_n^*\}$ is a a HS-basis.

When $p=1$, the argument is similar. 
\end{proof}

\begin{prop}\label{pp:4}
The Haar basis of $L^p[0,1]$ with $1<p\le2$ is an unconditional HS-basis.
\end{prop}
\begin{proof}
Let $\{h_n\}$ be the Haar system \cite{AK}, which is an unconditional basis in
$L^p[0,1]$ for $1<p<\infty$. Notice that the Haar system is not
normalized in $L^p[0,1]$ for $1<p<\infty$. To normalize them in $L^p$
one should take $h_n/\|h_n\|_p=|I_n|^{-1/p}h_n,$ where $I_n$
denotes the support of the Haar function $h_n.$ Then for
$1<p<\infty$, we have that the dual functionals associated to the
Haar system are given by
$h_n^*=\frac{1}{|I_n|}h_n, n\in\N.$
Thus, $\{h_n, h_n^*\}=\{h_n, |I_n|^{-1}h_n\}$ is an unconditional
basis system in $L^p$. By re-scaling, we have that
$\{|I_n|^{-1/p}h_n, |I_n|^{-1/q}h_n\}$ is a normalized unconditional basis
system in $L^p$ with $1/p+1/q=1$. By \cite{AO}, for $1<p<\infty$ there exist
constants $A_p,B_p$ such that, if $\{x_n\}$ is a normalized
$\lambda$-unconditional basic sequence in $L^p$, then
\begin{eqnarray*}
\lambda^{-1}\big(\sum_n|a_n|^p\big)^{1/p}\le\big\|\sum_n a_n
x_n\big\|_p\le \lambda B_p\big(\sum_n|a_n|^2\big)^{1/2}, \ \
\mbox{ if }\, 2\le p<\infty,
\end{eqnarray*}
\begin{eqnarray*}
(\lambda A_p)^{-1}\big(\sum_n|a_n|^2\big)^{1/2}\le\big\|\sum_n a_n
x_n\big\|_p\le \lambda\big(\sum_n|a_n|^p\big)^{1/p}, \ \ \mbox{
if }\, 1< p\le2.\end{eqnarray*} Thus, the operator $S$ defined by
\begin{eqnarray*}S:L^p\rightarrow L^q \mbox{ with } 1<p\le 2\le q<\infty
\mbox{ and } 1/p+1/q=1,\end{eqnarray*}
\begin{eqnarray*}
S\big(\sum_n a_n |I_n|^{-1/p}h_n\big)=\sum_n a_n |I_n|^{-1/q}h_n,
\quad \mbox{ for all } \sum_n a_n |I_n|^{-1/p}h_n \in L^p
\end{eqnarray*}
is well-defined and bounded. Clearly, we have
$S(|I_n|^{-1/p}h_n)=|I_n|^{-1/q}h_n$ for all $n\in\N.$ Thus,
$\{|I_n|^{-1/q}h_n\}$ in $L^q$ is an unconditional Hilbert-Schauder basis for
$L^p$.
\end{proof}

By Proposition \ref{pp:3} and \ref{pp:4}, $\ell_p$ $(1\le p\le 2)$ and
$L^p[0,1]$ $(1<p\le2)$ do have perfect HS-frames, normalized unconditional HS-bases $\{e_n,e_n^*\}$ with $\|e_n\|=\|e_n^*\|=1$ for all $n\in\N$, but the interesting thing is that $\ell_q$ with $2<q<\infty$ has no HS-frame $\{x_n,f_n\}$ with $\liminf_n \|f_n\|>0$.
\begin{lemma}\label{lem:1}
Let $\{x_n,f_n\}$ be a HS-frame of $X$ with the HSf-operator $S$. Then there is some $K>0$
such that $$K\cdot B_{X^*}\subset \overline{B_{X^*}\bigcap S(X)}^{w^*}.$$
\end{lemma}
\begin{proof} It is direct by the proof of Proposition 2.4 in \cite{Li}.
\end{proof}
\begin{prop}
$\ell_q$ with $2<q<\infty$ has no Hilbert-Schauder frame $\{x_n,f_n\}$ with $\liminf_n \|f_n\|>0$.
\end{prop}
\begin{proof}
Since $\ell_q$ is reflexive, by Lemma \ref{lem:1}, $$X^*=\overline{S(X)}^{w^*}=\overline{S(X)}^{w}=\overline{S(X)}.$$
Then, by Proposition \ref{pp:7}, we have $f=\sum_j f(x_j)f_j$ for all $f\in X^*$. By $\liminf_n \|f_n\|>0$, there
is a subsequence $\{f_{n_k}\}$ with $\inf_k \|f_{n_k}\|>0$. So it is easy to know that $\{x_{n_k}\}$ weakly converges to 0. By Pitt's theorem \cite{LT,AK}, the HSf-operator $S:\ell_q\rightarrow\ell_p$ with
$1/q+1/p=1$ and $p<q$ must be compact. Together with Proposition 7.6 in \cite{FHHSPZ}, we have
$f_{n_k}=S(x_{n_k})\rightarrow 0$, which leads to a contradiction.
\end{proof}

%
%

We give an isomorphic characterization of Hilbert spaces following the idea in \cite{HL}.

\begin{theorem}
A Banach space $X$ is isomorphic to a Hilbert space if and only if
it has a HS-frame $\{x_n\}$ with a HSf-operator $S$ such that $\{S x_n\}$ is also a
HS-frame of $\overline{[S x_n]}$.
\end{theorem}
\begin{proof}
The necessity is easy by standard Hilbert frame theory.

For sufficiency, assume that $(\{S x_n\},T)$ is a HSf-system of
$\overline{[S x_n]}$, then for all $x\in X$ and $f\in \overline{[S
x_n]}$,
\begin{eqnarray*}
(TS x)(f)&=&(TS\sum_n (S x_n)(x)\cdot x_n)(f)=(\sum_n (S
x_n)(x)\cdot T S x_n)(f)\\&=&\sum_n (S x_n)(x)\cdot (T S
x_n)(f)=\sum_n (T S x_n)(f)\cdot (S x_n)(x)\\&=&(\sum_n (T S
x_n)(f)\cdot S x_n)(x)=f(x).
\end{eqnarray*}
By Proposition 2.4 in \cite{Li}, $\overline{[S x_n]}$ is a norming subspace of
$X^*$, it follows that $$\|TSx\|_{\overline{[S x_n]}^*}
=\sup_{f\in \overline{[S x_n]}, \|f\|\le 1}|(TSx)(f)|
=\sup_{f\in \overline{[S x_n]}, \|f\|\le 1}|f(x)|
\approx\|x\|_X.$$ Thus, $TS$ is an isomorphic embedding. By Proposition
\ref{pp:2}, we have $S=A^*A$. Then $\|A x\|\ge \|TA^*\|^{-1}\|TS x\|
\ge K \|x\|$ for some $K>0$, that is, $A:X\rightarrow\ell_2$ is an
isomorphic embedding, which implies that $X$ is isomorphic to a
Hilbert space.
\end{proof}


%
%
%
%
%
\vskip1.5cm
\noindent\textbf{Acknowledgment.}
This work was completed when I visited the University of Texas at Austin and Texas A\&M University.
I would like to thank David Larson and E. Odell for the invitation and great
help. 

The author also expresses his appreciation to Dr. David Larson
for many very helpful comments regarding frame theory and operator algebra.
The interested readers should consult the papers \cite{CHL,HL,HLLL}.

\end{document}